\batchmode
\makeatletter
\def\input@path{{\string"/Users/russw/Documents/Research/mypapers/Results on the regularity of square-free monomial ideals/\string"/}}
\makeatother
\documentclass[12pt,english]{amsart}
\usepackage[T1]{fontenc}
\usepackage[latin9]{inputenc}
\usepackage{geometry}
\usepackage{color}
\usepackage{babel}
\usepackage{url}
\usepackage{enumitem}
\usepackage{amsthm}
\usepackage{amssymb}
\usepackage{graphicx}
\usepackage[unicode=true,
 bookmarks=true,bookmarksnumbered=true,bookmarksopen=true,bookmarksopenlevel=1,
 breaklinks=false,pdfborder={0 0 1},backref=false,colorlinks=true]
 {hyperref}
\hypersetup{pdftitle={Regularity of Edge Ideals},
 pdfauthor={Huy T\`ai H\`a and Russ Woodroofe},
 pdfsubject={Combinatorial Commutative Algebra},
 pdfkeywords={edge ideal, regularity, hypergraphs},
 linkcolor=blue,pdfproducer={Latex with hyperref},letterpaper=true,pdfcreator={latex->dvips->ps2pdf}}
\usepackage{breakurl}

\makeatletter
\numberwithin{equation}{section}
\numberwithin{figure}{section}
\theoremstyle{plain}
\newtheorem{thm}{\protect\theoremname}[section]
  \theoremstyle{definition}
  \newtheorem{defn}[thm]{\protect\definitionname}
  \theoremstyle{plain}
  \newtheorem{lem}[thm]{\protect\lemmaname}
  \theoremstyle{remark}
  \newtheorem{rem}[thm]{\protect\remarkname}
  \theoremstyle{definition}
  \newtheorem{example}[thm]{\protect\examplename}
  \theoremstyle{plain}
  \newtheorem{cor}[thm]{\protect\corollaryname}
  \theoremstyle{plain}
  \newtheorem{prop}[thm]{\protect\propositionname}
 \newlist{casenv}{enumerate}{4}
 \setlist[casenv]{leftmargin=*,align=left,widest={iiii}}
 \setlist[casenv,1]{label={{\itshape\ \casename} \arabic*.},ref=\arabic*}
 \setlist[casenv,2]{label={{\itshape\ \casename} \roman*.},ref=\roman*}
 \setlist[casenv,3]{label={{\itshape\ \casename\ \alph*.}},ref=\alph*}
 \setlist[casenv,4]{label={{\itshape\ \casename} \arabic*.},ref=\arabic*}

\makeatother

  \providecommand{\corollaryname}{Corollary}
  \providecommand{\definitionname}{Definition}
  \providecommand{\examplename}{Example}
  \providecommand{\lemmaname}{Lemma}
  \providecommand{\propositionname}{Proposition}
  \providecommand{\remarkname}{Remark}
 \providecommand{\casename}{Case}
\providecommand{\theoremname}{Theorem}

\begin{document}
\global\long\def\E{\mathcal{E}}

\global\long\def\H{\mathcal{H}}

\global\long\def\Y{\mathcal{Y}}

\global\long\def\reg{\operatorname{reg}}

\global\long\def\tor{\operatorname{Tor}}

\global\long\def\indmatch{\nu_{\mathrm{ind}}}

\global\long\def\mmmatch{\nu_{\mathrm{min}}}

\global\long\def\link{\operatorname{link}}

\global\long\def\del{\operatorname{del}}

\global\long\def\susp{\operatorname{susp}}

\title{Results on the regularity of square-free monomial ideals}

\author{Huy Tài Hà}

\address{Department of Mathematics\\
Tulane University\\
6823 St. Charles Ave., New Orleans, LA 70118}

\urladdr{\url{http://www.math.tulane.edu/\~tai/} }

\email{tai@math.tulane.edu}

\thanks{The first author is partially supported by NSA grant H98230-11-1-0165.}

\author{Russ Woodroofe}

\address{Department of Mathematics \& Statistics\\
Missisissippi State University\\
Starkville, MS 39762}

\urladdr{\url{http://rwoodroofe.math.msstate.edu/}}

\email{rwoodroofe@math.msstate.edu}
\begin{abstract}
In a 2008 paper, the first author and Van Tuyl proved that the regularity
of the edge ideal of a graph $G$ is at most one greater than the
matching number of $G$. In this note, we provide a generalization
of this result to any square-free monomial ideal. We define a 2-collage
in a simple hypergraph to be a collection of edges with the property
that for any edge $E$ of the hypergraph, there exists an edge $F$
in the 2-collage such that $|E\setminus F|\le1$. The Castelnuovo-Mumford
regularity of the edge ideal of a simple hypergraph is bounded above
by a multiple of the minimum size of a 2-collage. We also give a recursive
formula to compute the regularity of a vertex-decomposable hypergraph.
Finally, we show that regularity in the graph case is bounded by a
certain statistic based on maximal packings of nondegenerate star
subgraphs. 
\end{abstract}
\maketitle

\section{Introduction}

Let $k$ be a field. There is a natural correspondence between square-free
monomial ideals in $R=k[x_{1},\dots,x_{n}]$ and simple hypergraphs
over the vertices $V=\{x_{1},\dots,x_{n}\}$. This correspondence
has evolved to be an instrumental tool in an active research program
in combinatorial commutative algebra --- we recommend any of \cite{Herzog/Hibi:2011,Miller/Sturmfels:2005,Morey/Villarreal:2012,Villarreal:2001}
for an overview. One goal of this research program is to relate algebraic
properties and invariants of a square-free monomial ideal to combinatorial
properties and statistics of the corresponding simple hypergraph.
In this note, we will examine the Castelnuovo-Mumford regularity of
such ideals, which has been previously studied in work including \cite{Bouchat/Ha/OKeefe:2011,Dao/Huneke/Schweig:2013,Francisco/Ha/VanTuyl:2009,Kimura:2012,Kummini:2009,Nevo:2011,VanTuyl:2009}. 

In \cite[Theorem 6.7]{Ha/VanTuyl:2008}, the first author and Van
Tuyl showed that the regularity of the edge ideal of a graph $G$
is at most one greater than the matching number of $G$. Indeed, it
follows from their proof and was explicitly noticed in \cite{Woodroofe:2010UNPc}
that the regularity of the edge ideal of a graph is at most one greater
than the minimum size of a maximal matching. The first goal of this
note is extend this result to the edge ideal of a hypergraph, i.e.,
to any square-free monomial ideal.

Our first bound for the Castelnuovo-Mumford regularity of a square-free
monomial ideal is based on the notion of 2-collage. If $\mathcal{H}=(V,\E)$
is a hypergraph, then a \emph{2-collage for $\H$} is a subset $\mathcal{C}$
of the edges with the property that for each $E\in\E$ we can delete
a vertex $v$ so that $E\setminus\{v\}$ is contained in some edge
of $\mathcal{C}$. For uniform hypergraphs, the condition for a collection
$\mathcal{C}$ of the edges to be a 2-collage is equivalent to requiring
that for any edge $E$ not in $\mathcal{C}$, there exists $F\in\mathcal{C}$
such that the symmetric difference of $E$ and $F$ consists of exactly
two vertices. When $\H$ is a graph, it is straightforward to see
that for any minimal 2-collage, there is a maximal matching of the
same or lesser cardinality. Our first main result is:
\begin{thm}
\label{intro.thm.collage-uniform} Let $\H$ be a simple $d$-uniform
hypergraph with edge ideal $I\subseteq R$, and let $c$ be the minimum
size of a 2-collage in $\H$. Then $\reg(R/I)\le(d-1)c$. 
\end{thm}
Indeed, Theorem \ref{intro.thm.collage-uniform} will follow from
the following more general result:
\begin{thm}
\label{intro.thm.collage} Let $\H$ be a simple hypergraph with edge
ideal $I\subseteq R$, and let $\{E_{1},\dots,E_{c}\}$ be a 2-collage
in $\H$. Then $\reg(R/I)\leq\sum_{i=1}^{c}\left(\left|E_{i}\right|-1\right)$.
\end{thm}
In the case where our hypergraph $\H$ is not a graph, a minimum 2-collage
in $\H$ is generally bigger than the \emph{minimax matching number}
(that is, the minimum size of a maximal matching). We shall see in
Example \ref{ex.matching} that even in the uniform case, the bound
in Theorem \ref{intro.thm.collage-uniform} is no longer true if we
replace $c$ by the minimax matching number $m$ of $\H$, and in
fact that $\reg(R/I)$ can be arbitrary larger than $(d-1)m$. If
$\H$ is a graph, the minimum size of a 2-collage is easily seen to
be the minimax matching number, so Theorem \ref{intro.thm.collage-uniform}
restricted to graphs recovers \cite[Theorem 6.7]{Ha/VanTuyl:2008}
and \cite[Theorem 11 and discussion following]{Woodroofe:2010UNPc}.

Upper bounds are most interesting when compared with lower bounds,
and while hypergraph matchings do not in general seem to give any
upper bound for regularity, a related notion will give a lower bound.
We call a collection $\{E_{1},\dots,E_{\ell}\}$ of edges in $\H$
an \textit{induced matching} if they form a matching in $\H$ (i.e.,
they are pairwise disjoint), and they are exactly the edges of the
induced subhypergraph of $\H$ over the vertices contained in $\bigcup_{i=1}^{\ell}E_{i}$.
The \textit{induced matching number} of $\H$, denoted by $\indmatch(\H)$,
is the maximum size of an induced matching in $\H$. The following
was proved in \cite[Theorem 6.5]{Ha/VanTuyl:2008} for properly connected
simple hypergraphs, and was extended to all simple hypergraphs in
\cite[Corollary 3.9]{Morey/Villarreal:2012}:
\begin{thm}
\emph{\cite[Theorem 6.5]{Ha/VanTuyl:2008}\cite[Corollary 3.9]{Morey/Villarreal:2012}}
\label{intro.thm.lb} Let $\H$ be a simple hypergraph with edge ideal
$I\subseteq R$, and let $\{E_{1},\dots,E_{\ell}\}$ be an induced
matching in $\H$. Then $\reg(R/I)\geq\sum_{i=1}^{\ell}\left(\left|E_{i}\right|-1\right)$.
\end{thm}
For ease of comparison with the upper bound of Theorem \ref{intro.thm.collage-uniform},
we restate Theorem \ref{intro.thm.lb} in the case where $\H$ is
uniform:
\begin{thm}
\label{intro.thm.lb-uniform}Let $\H$ be a simple $d$-uniform hypergraph
with edge ideal $I\subseteq R$. Then $\reg(R/I)\ge(d-1)\indmatch(\H).$
\end{thm}
A second goal of this note is to describe the regularity of \emph{vertex-decomposable
}graphs, a class of graphs that has garnered considerable recent attention
\cite{VanTuyl:2009,Woodroofe:2009a}. In particular, the quotient
ring associated to the edge ideal of a vertex-decomposable graph or
hypergraph is sequentially Cohen-Macaulay \cite{Bjorner/Wachs:1997}.
We give a recursive formula to compute regularity of any vertex-decomposable
hypergraph:
\begin{thm}
\label{intro.thm.vd}Let $\H$ be a vertex-decomposable simple hypergraph
with edge ideal $I\subseteq R$, and with $v$ the initial vertex
in the shedding order. Then $\reg(R/I)=\max\left\{ \reg(I:v)+1,\reg(I,v)\right\} .$ 
\end{thm}
We will actually prove a slightly more general (if somewhat technical)
result, which weakens the vertex-decomposability condition to a sequentially
Cohen-Macaulay condition on the vertex deletion subcomplex. We state
this precisely as Theorem~\ref{thm:vd.reg.generalized} below.

Note that Dao, Huneke and Schweig observed \cite{Dao/Huneke/Schweig:2013}
that $\reg(I)$ is equal to one of $\reg(I:v)+1$ or $\reg(I,v)$
for any hypergraph $\H$ and vertex $v$. Our result is that $\reg(I)$
is always the larger of the two in the case of a vertex-decomposable
hypergraph and its shedding vertex.

\medskip{}
Our third goal will be to give a new upper bound for the regularity
of any graph. Our upper bound will be based on a certain packing-type
invariant, as follows. The \emph{closed neighborhood of $x$} in a
graph $G$, denoted $N_{G}[x]$, is the subset of vertices consisting
of $x$ and all of its neighbors. A closely related notion is the
\emph{star} at $x$, which is the subgraph on $N_{G}[x]$ with edge
set consisting of all edges of $G$ incident to $x$. We say that
a star is \emph{nondegenerate} if $\deg x>1$, so that the star doesn't
consist of a single vertex or single edge.

Our upper bound will be based on packing nondegenerate stars into
$G$. We say a set of stars is \emph{center-separated} if the center
of a star and at least two of its neighbors are not contained in any
other star. After deleting the vertices of the stars in a maximal
center-separated star packing $\mathcal{P}$, an induced matching
of $G$ will remain. Let $\zeta_{\mathcal{P}}$ be the number of stars
in the packing plus the number of edges in the induced matching remainder,
and let $\zeta(G)$ be the maximum $\zeta_{\mathcal{P}}$ over all
maximal center-separated packings of nondegenerate stars. Our third
main theorem is: 
\begin{thm}
\label{intro.thml.packing} Let $G$ be a graph with edge ideal $I\subseteq R$.
Then $\reg(R/I)\leq\zeta(G)$.
\end{thm}
It is clear that $\zeta(G)$ is at most the matching number of $G$,
so Theorem \ref{intro.thml.packing} is another generalization of
the matching upper bound of Hà and Van Tuyl. Theorem \ref{intro.thml.packing}
also improves on bounds of Moradi and Kiani \cite{Moradi/Kiani:2010}
proved with an additional assumption of vertex-decomposability and/or
shellability, as we will discuss in Remarks \ref{Rem:MoradiKianiShell}
and \ref{Rem:MoradiKianiVD}. Theorem \ref{intro.thml.packing} is
proved inductively, and the main step in the induction (Lemma \ref{lem:DegGr1Removal})
may be of independent interest.

\medskip{}

This paper is organized as follows. In Section \ref{sec:Notation-and-terminology},
we shall collect the necessary notations and terminology. The proof
of Theorem \ref{intro.thm.collage} is given in Section \ref{sec:Proofs}.
The main tool for this theorem is a result of Kalai and Meshulam \cite{Kalai/Meshulam:2006}
bounding the regularity of the sum of square-free monomial ideals.
In Section \ref{sec:VD}, we use the Stanley-Reisner face ring correspondence
and the combinatorial topology of simplicial complexes to prove Theorem~\ref{intro.thm.vd}.
Finally, in Section~\ref{sec:Packing-bound}, we prove Theorem \ref{intro.thml.packing}.

\section{\label{sec:Notation-and-terminology}Notation and terminology}

\subsection{Hypergraphs and edge ideals}

A\textit{ hypergraph} $\H$ consists of a set $V=\{x_{1},\dots,x_{n}\}$,
called \textit{vertices}; and a collection $\E$ of nonempty subsets
of $V$, called \textit{edges}. We will use $V(\H)$ and $\E(\H)$
to denote the sets of vertices and edges, respectively, of $\H$.
A hypergraph is \emph{simple} if there are no nontrivial containments
among the edges (i.e., if $E\subseteq E'$ are edges then $E=E'$).
All hypergraphs discussed in this paper will be simple. Simple hypergraphs
have been studied under several other names, including ``clutter''
and ``Sperner system''. An important family of simple hypergraphs
are \emph{$d$-uniform hypergraphs}, in which every edge contains
exactly $d$ vertices.

Let $k$ be a field, and identify the vertices in $V$ with the variables
in a polynomial ring $R=k[x_{1},\dots,x_{n}]$. The following construction
gives a one-to-one correspondence between square-free monomial ideals
in $R$ and simple hypergraphs over $V$: 
\begin{defn}
Let $\H=(V,\E)$ be a simple hypergraph. For a subset $E\subseteq V$,
let $x^{E}$ denote the monomial $\prod_{x_{i}\in E}x_{i}$. The \textit{edge
ideal} of $\H$ is the square-free monomial ideal 
\[
I(\H)=\big(x^{E}\,\big|\, E\in\E\big)\subseteq R.
\]

\end{defn}
Certain substructures of a hypergraph will be important to us. If
$\H$ is a hypergraph with an edge $E$, then $\H\setminus E$ will
denote the hypergraph obtained from $\H$ by removing $E$ from the
edge set. The \emph{induced subhypergraph} of $\H$ on a subset $W$
of the vertex set is the hypergraph over vertex set $W$ with edge
set consisting of all edges of $\H$ that are contained in $W$.

\subsection{Simplicial complexes and vertex-decomposability}

The edge ideal $I(\H)$ is a square-free monomial ideal, so it can
also be viewed as the Stanley-Reisner ideal of a simplicial complex
as follows: 
\begin{defn}
We call a collection $\mathcal{B}$ of the vertices of hypergraph
$\H$ an \textit{independent set} if there is no edge $E$ in $\H$
such that $E\subseteq\mathcal{B}$. The \textit{independence complex}
of $\H$, denoted by $\Delta(\H)$, is the simplicial complex whose
faces consist of all independent sets in $\H$. 
\end{defn}
It is immediate from the definitions that if $I_{\Delta}$ denotes
the Stanley-Reisner ideal of $\Delta$, then $I(\H)=I_{\Delta(\H)}.$

If $v$ is a vertex of the simplicial complex $\Delta$, then the
\emph{deletion} of $v$ from \emph{$\Delta$}, denoted by $\del_{\Delta}(v)$,
is the simplicial complex over the vertex set $V\setminus\{v\}$ with
faces $\{\sigma\,|\,\sigma\in\Delta,v\notin\sigma\}$. An \emph{induced
subcomplex} of $\Delta$ is obtained by (successively) deleting a
set of vertices. The \emph{link }of $v$ in\emph{ $\Delta$}, denoted
by $\link_{\Delta}v$, is the subcomplex of $\del_{\Delta}v$ with
faces $\{\sigma\,|\,\sigma\in\del_{\Delta}v,v\cup\sigma\in\Delta\}$.

Algebraically, we have $I_{\del_{\Delta}v}=(I_{\Delta},v)$, while
$I_{\link_{\Delta}v}=(I_{\Delta}:v,v)$. In particular, since $v$
does not appear in any monomial in $I:v$, we will see that $\reg(I_{\link_{\Delta}v})=\reg(I:v)$.
\medskip{}

A simplicial complex $\Delta$ is recursively defined to be \emph{vertex-decomposable}
if either 
\begin{itemize}
\item [(a)]$\Delta$ is a simplex, or 
\item [(b)]there exists a vertex $v$ such that both $\del_{\Delta}(v)$
and $\link_{\Delta}(v)$ are vertex-decomposable, and the facets of
$\del_{\Delta}(v)$ are facets of $\Delta$. 
\end{itemize}
A vertex satisfying the condition in (b) is called a \emph{shedding
vertex}, and the recursive choice of vertices is called a \emph{shedding
order}. When it causes no confusion, we will call a simple hypergraph
$\H$ \textit{vertex-decomposable} if its independence complex $\Delta(\H)$
is vertex-decomposable. \medskip{}

A complex is \emph{shellable} if there is an ordering of its facets
obeying certain restrictions, the precise details of which will not
be important for us. It is well-known that 
\[
\Delta\mbox{ vertex-decomposable}\implies\Delta\mbox{ shellable}\implies\Delta\mbox{ sequentially Cohen-Macaulay.}
\]
 For additional background on the combinatorics of simplicial complexes,
including vertex-decomposability and shellability, we refer to e.g.
\cite{Bjorner:1995,Kozlov:1997}; for background on the connection
with commutative algebra, we refer to \cite{Miller/Sturmfels:2005,Stanley:1996}.

\subsection{Regularity}

Recall that the \emph{Castelnuovo-Mumford regularity} (or just \emph{regularity)}
of an $R$-module $M$ can be defined as 
\[
\reg(M)=\max_{i}\{\max\{j\,|\,\tor_{i}^{R}(M,k)_{j}\not=0\}-i\}.
\]
For an overview of and background on Castelnuovo-Mumford regularity,
we refer to the recent survey article \cite{Chardin:2007}.

We observe that for $R$ a polynomial ring, $\reg(I)=\reg(R/I)+1$;
thus, it is equivalent to study the regularity of the edge ideal or
the corresponding quotient ring. Our notation is a bit careless about
what polynomial ring we are working over: this is justified, as if
$S$ is any polynomial ring over $k$ containing $R$ (with additional
variables not appearing in $I$), then $\reg(R/I)=\reg(S/I)$. Of
our main theorems, only Theorem \ref{thm:vd.reg.generalized} depends
on the choice of the field $k$, and this only insofar as the sequentially
Cohen-Macaulay property may depend on $k$.

All simplicial homology will be taken over the same coefficient field
$k$ as $R$, and we suppress the field from our notation. By the
Universal Coefficient Theorem, the simplicial homology $\tilde{H}_{i}(\Delta;k)$
and cohomology $\tilde{H}^{i}(\Delta;k)$ are isomorphic when $k$
is a field, so we can work with whichever is more convenient.

\medskip{}

In Section \ref{sec:VD}, we will find it helpful to work with regularity
through independence complexes and the Stanley-Reisner correspondence.
We summarize the connection:
\begin{lem}
\emph{\label{lem:RegFromTopology}}For a simplicial complex $\Delta$,
the following are equivalent:
\begin{enumerate}
\item $\reg(R/I_{\Delta})\geq d$.
\item $\tilde{H}_{d-1}(\Delta[S])\neq0$ for some $S\subseteq V$, where
$\Delta[S]$ denotes the induced subcomplex on $S$.
\item $\tilde{H}_{d-1}(\link_{\Delta}\sigma)\neq0$ for some face $\sigma$
of $\Delta$.
\end{enumerate}
\end{lem}
\noindent We briefly sketch a proof: The equivalence of (1) and (2)
follows directly from the Betti number characterization of regularity
\cite{Chardin:2007}, together with Hochster's formula (as stated
in \cite[Corollary 5.12]{Miller/Sturmfels:2005}). The equivalence
of (1) and (3) follows directly from the local cohomology characterization
of regularity \cite{Chardin:2007}, together with the fact that $H_{\mathfrak{m}}^{i}(R/I_{\Delta})_{-\sigma}\cong\tilde{H}^{i-\left|\sigma\right|-1}(\link_{\Delta}\sigma)$
\cite[Chapter 13.2]{Miller/Sturmfels:2005}. 

Kalai and Meshulam \cite[Proposition 3.1]{Kalai/Meshulam:2006} have
also given a direct proof of the equivalence of (2) and (3).
\begin{rem}
In some sources in the topological combinatorics literature \cite{Alon/Kalai/Matousek/Meshulam:2002,Kalai/Meshulam:2006},
complexes $\Delta$ with $\reg(R/I_{\Delta})\geq d$ are called \emph{$d$-Leray},
and $\reg(R/I_{\Delta})$ is referred to as the \emph{Leray number}
of $\Delta$.
\end{rem}
The following well-known lemma follows directly from characterization
(2) of Lemma \ref{lem:RegFromTopology}, and tells us that regularity
may be regarded as giving a measure of the complexity of $\H$.
\begin{lem}
\label{lem:RegShrinksInSubhypergraphs} Let $\H$ be a simple hypergraph.
Then $\reg I(\H)\geq\reg I(\H')$ for any induced subhypergraph $\H'$
of $\H$. 
\end{lem}
A similar result holds for links:
\begin{lem}
\label{lem:RegShrinksInLinks} Let $\Delta$ be a simplicial complex.
Then $\reg(R/I_{\Delta})\geq\reg(R/I_{\link_{\Delta}\sigma})$ for
any face $\sigma$ of $\Delta$. 
\end{lem}

\section{\label{sec:Proofs} Regularity and collages}

In this section, we prove Theorem \ref{intro.thm.collage}, bounding
the regularity of a square-free monomial ideal with a collage. For
completeness, and because the proof is short, we begin by proving
the bound from below of Theorem \ref{intro.thm.lb}:
\begin{proof}[Proof of Theorem \ref{intro.thm.lb}]
 Let $\{E_{1},\dots,E_{\ell}\}$ be an induced matching in $\H$
and let $\H'$ be the induced subhypergraph of $\H$ over the vertices
contained in $\bigcup_{i=1}^{\ell}E_{i}$. Since the $E_{i}$'s are
pairwise disjoint, the regularities add, so we have $\reg(R/I(\H'))=\sum_{i=1}^{\ell}\left(\left|E_{i}\right|-1\right)$,
and the result follows by Lemma \ref{lem:RegShrinksInSubhypergraphs}. 
\end{proof}
In our proof of Theorem \ref{intro.thm.collage}, we shall make use
of the following theorem of Kalai and Meshulam.
\begin{thm}
\cite[Theorem 1.4]{Kalai/Meshulam:2006} \label{thm.KM} Let $\H$
and $\H_{1},\dots,\H_{s}$ be simple hypergraphs over the same vertex
set $V$ such that $\E(\H)=\bigcup_{i=1}^{s}\E(\H_{i})$. Then 
\[
\reg(R/I(\H))\le\sum_{i=1}^{s}\reg(R/I(\H_{i})).
\]
\end{thm}
\begin{rem}
Kalai and Meshulam gave a topological proof of Theorem \ref{thm.KM}
via the correspondence in Lemma \ref{lem:RegFromTopology}. Theorem
\ref{thm.KM} was later extended to arbitrary (not necessarily square-free)
monomial ideals by Herzog \cite{Herzog:2007}, who used algebraic
techniques.
\end{rem}
We will need a technical lemma. If $E$ is an edge of $\H$, then
let $\H_{E}$ be the hypergraph whose edge set consists of the minimal
(under inclusions) members of $\{E'\cup E\,:\, E'\neq E\mbox{ is an edge of }\H\}$. 
\begin{lem}
\label{lem.edgebyedge} Let $\H$ be a hypergraph with at least two
edges, $E$ be an edge of $\H$, and $\H_{E}$ be as in the preceding
paragraph. Then $\reg(I(\H))\leq\max\left\{ \reg(I(\H\setminus E),\reg(I(\H_{E}))-1\right\} $.
\end{lem}
Lemma \ref{lem.edgebyedge} follows from a long exact sequence argument,
arising from the fact that $I(\H_{E})=(x^{E})\cap I(\H\setminus E)$
and from the short exact sequence $(x^{E})\cap I(H\setminus E)\rightarrow(x^{E})\oplus I(\H\setminus E)\rightarrow I(\H)$.
(More details can be found in \cite[Theorem 6.2]{Ha/VanTuyl:2008}.)

We are now ready to prove Theorem \ref{intro.thm.collage}. We begin
with the case where $c=1$: 
\begin{lem}
\label{lem.star} If $\H=(V,\E)$ is a hypergraph such that $\{E_{0}\}$
is a 2-collage for $\H$, then $\reg(I(\H))=\vert E_{0}\vert$.\end{lem}
\begin{proof}
We proceed by induction on the number of edges in $\mathcal{H}$.
If $\mathcal{H}$ has a single edge, then $\reg(I(\H))=\reg(x^{E_{0}})=\vert E_{0}\vert$,
as desired. Otherwise, let $E\neq E_{0}$ be any other edge of $\H$,
and apply Lemma \ref{lem.edgebyedge} to give that $\reg(I(\H))\leq\max\left\{ \reg(I(\H\setminus E),\reg(I(\H_{E})-1\right\} $.
But then $\{E\cup E_{0}\}$ is a $2$-collage for $\H_{E}$, and the
result follows from induction together with the observation that $\left|E\cup E_{0}\right|=\left|E\right|+1$.
\end{proof}

\begin{proof}[Proof of Theorem \ref{intro.thm.collage}]
 Let $\{E_{1},\dots,E_{c}\}$ be a 2-collage, and for each $i$ let
$\H_{i}$ be the simple hypergraph consisting of all $E$ with $E\setminus\{v\}\subseteq E_{i}$
(as in the definition of 2-collage). By construction we have that
$\E(\H)=\bigcup_{i=1}^{\ell}\E(\mathcal{H}_{i})$ and that each $\mathcal{H}_{i}$
meets the conditions of Lemma \ref{lem.star}. The result now follows
from Theorem \ref{thm.KM}. 
\end{proof}
Recall that a \textit{matching} in a simple hypergraph $\H$ is a
collection of pairwise disjoint edges, and the \textit{matching number}
of $\H$ (denoted by $\nu(\H)$) is the maximum size of any matching
in $\H$. We similarly denote the \emph{minimax matching number} (the
minimum size of a maximal matching) as $\mmmatch(\H)$. In \cite{Ha/VanTuyl:2008},
it was essentially shown that if $G$ is a graph then 
\[
\reg(R/I(G))\le\mmmatch(G).
\]
In the next example, we shall see that the analogous bound of $(d-1)\nu_{\min}(\H)$
no longer holds for hypergraphs, thus that the 2-collage in the statement
of Theorem \ref{intro.thm.collage} cannot be replaced by a matching
(even in the uniform case). 

\begin{figure}
\includegraphics{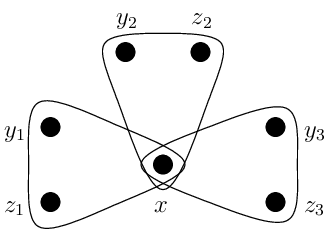}

\caption{\label{fig:MatchingBdFailure}A 3-uniform hypergraph with regularity
greater than $(3-1)\cdot\nu$.}
\end{figure}

\begin{example}
\label{ex.matching}For $s>1$, consider the hypergraph $\H_{s}$
on the vertex set $\{x,y_{1},\dots,y_{s},z_{1},\dots,z_{s}\}$ with
edges $\{x,y_{i},z_{i}\}$ (for $i=1,\dots,s$). Figure~\ref{fig:MatchingBdFailure}
illustrates $\H_{3}$. We have that the matching number and minimax
matching number of $\H_{s}$ are both 1. On the other hand, it is
straightforward to compute that $\reg(R/I(\H_{s}))=s+1$, which can
be taken to be arbitrarily far from $(3-1)\nu(\H_{s})=(3-1)\mmmatch(\H_{s})=2$. 
\end{example}
It is easy to see that the bounds in Theorems and \ref{intro.thm.collage}
and \ref{intro.thm.lb} are not sharp. Indeed \cite{Woodroofe:2010UNPc}
observes that the disjoint union of cyclic graphs can give arbitrarily
large differences between the induced matching number, regularity,
and matching number.

\medskip{}
 We close this section by stating an equivalent form to Theorem \ref{intro.thm.collage}
in somewhat different language. The definition of matching essentially
calls for a set of edges that are as separated as possible. A notion
of separation that allows us to interpolate between a hypergraph matching
and an arbitrary set of edges is as follows:
\begin{defn}
Let $t$ be a positive integer. Two distinct edges $E$ and $F$ are
said to be \emph{$t$-separated} if either $|E\setminus F|\ge t$
or $|F\setminus E|\ge t$. 
\end{defn}
Thus, a matching in a graph is a collection of pairwise $2$-separated
edges; more generally a matching in any $d$-uniform hypergraph is
a collection of pairwise $d$-separated edges. It is immediate from
definition that a maximal collection of pairwise $2$-separated edges
forms a 2-collage.
\begin{cor}
Let $\H$ be a simple hypergraph with edge ideal $I\subseteq R$,
and let $\{E_{1},\dots,E_{c}\}$ be a maximal collection of pairwise
$2$-separated edges in $\H$. Then $\reg(R/I)\leq\sum_{i=1}^{c}\left(\left|E_{i}\right|-1\right)$.
\end{cor}

\section{\label{sec:VD} Regularity in a vertex-decomposable simplicial complex}

In this section, we will prove Theorem \ref{intro.thm.vd}, and stronger
versions thereof. Our main tool will be the combinatorial topology
of simplicial complexes, together with the characterization of regularity
from Lemma \ref{lem:RegFromTopology}. Throughout this section and
the next we will freely abuse notation to write $\reg\Delta$ for
$\reg(R/I_{\Delta})$.\medskip{}

As mentioned previously, a Mayer-Vietoris argument yields:
\begin{lem}
\label{lem:RegAlternative} \cite[Lemma 2.10]{Dao/Huneke/Schweig:2013}
If $\H$ is a simple hypergraph with edge ideal $I\subseteq R$, and
$v$ is any vertex of $\H$, then $\reg(I)$ is either $\reg(I:x)+1$
or $\reg(I,x)$.
\end{lem}
Restated in terms of the independence complex $\Delta$ of $\H$,
Lemma \ref{lem:RegAlternative} says that either $\reg\Delta=\reg(\link_{\Delta}v)+1$
or else $\reg\Delta=\reg(\del_{\Delta}v)$. This is particularly intuitive
from a geometric perspective, where it essentially says that either
$v$ is contained in an (induced) homology cycle of highest possible
dimension, or else that some such homology cycle avoids $v$.

It follows immediately that 
\begin{equation}
\reg\Delta\leq\max\{\reg(\link_{\Delta}v)+1,\reg(\del_{\Delta}v)\}.\label{eq:RegRecursiveMax}
\end{equation}
 It is clear from Lemma \ref{lem:RegFromTopology} part (2) that $\reg(\del_{\Delta}v)\leq\reg\Delta$,
hence, if the max of (\ref{eq:RegRecursiveMax}) is obtained on $\reg(\del_{\Delta}v)$,
then $\reg\Delta=\reg(\del_{\Delta}v)$. On the other hand, while
Lemma \ref{lem:RegShrinksInLinks} gives that $\reg\Delta$ may not
be smaller than $\reg(\link_{\Delta}v)$, taking $\Delta$ to be a
cone with apex $v$ over any complex shows that $\reg\Delta$ may
equal $\reg(\link_{\Delta}v)$, hence may be strictly less than $\reg(\link_{\Delta}v)+1$.
\medskip{}

Theorem \ref{intro.thm.vd} gives a concrete set of circumstances
which guarantee that a homology $n$-cycle in $\link_{\Delta}v$ lifts
to a homology $(n+1)$-cycle in $\Delta$. We will prove the following
generalization:
\begin{thm}
\label{thm:vd.reg.generalized}If $v$ is a shedding vertex for a
simplicial complex $\Delta$ such that $\Delta\setminus v$ is sequentially
Cohen-Macaulay, then $\reg\Delta=\max\{\reg(\Delta\setminus v),\reg(\link_{\Delta}v)+1\}$.
\end{thm}
We will prove Theorem \ref{thm:vd.reg.generalized} via several lemmas.
\begin{lem}
If $\sigma$ is a face and $v$ a shedding vertex of a simplicial
complex $\Delta$ such that $v\notin\sigma$, then $v$ is a shedding
vertex for $\link_{\Delta}\sigma$.\end{lem}
\begin{proof}
Immediate from definition: if $\tau$ is a face containing $\sigma\cup v$,
then (since $v$ is a shedding vertex) there is a vertex $w$ such
that $(\tau\setminus v)\cup w$ is a face. As $\sigma\subseteq(\tau\setminus v)\cup w$,
we obtain the shedding vertex condition for $v$ in $\link_{\Delta}\sigma$.
\end{proof}
We denote by $\Gamma^{[n]}$ the \emph{pure $n$-skeleton} of a simplicial
complex $\Gamma$, consisting of all faces contained in a face of
dimension $n$. 

The following lemma will be the core of the proof of Theorem \ref{thm:vd.reg.generalized}. 
\begin{lem}
\label{lem:Lift}Let $\Gamma$ be a simplicial complex, and suppose
that that $\tilde{H}_{n}(\link_{\Gamma}v)\neq0$. If $\left(\link_{\Gamma}v\right)^{[n]}$
is contained in a subcomplex $\Gamma_{0}$ of $\del_{\Gamma}v$ with
$\tilde{H}_{n}(\Gamma_{0})=0$, then $\tilde{H}_{n+1}(\Gamma)\neq0$.\end{lem}
\begin{proof}
We use the exactness of the Mayer-Vietoris sequence 
\[
\cdots\rightarrow\tilde{H}_{n+1}(\Gamma)\overset{g}{\rightarrow}\tilde{H}_{n}(\link_{\Gamma}v)\overset{f}{\rightarrow}\tilde{H}_{n}(\del_{\Gamma}v)\rightarrow\cdots.
\]
The result then follows by noting that the map $f$ is induced from
the inclusion map.
\end{proof}
Duval \cite[Theorem 3.3]{Duval:1996} proved that a complex $\Delta$
is sequentially Cohen-Macaulay if and only if $\Delta^{[n]}$ is Cohen-Macaulay
for all $n$. We use this to prove:
\begin{cor}
\label{cor:LiftSCM} If $\del_{\Gamma}v$ is a sequentially Cohen-Macaulay
complex, and $v$ is a shedding vertex of $\Gamma$ with $\tilde{H}_{n}(\link_{\Gamma}v)\neq0$,
then $\tilde{H}_{n+1}(\Gamma)\neq0$.\end{cor}
\begin{proof}
By definition of shedding vertex, $\left(\link_{\Gamma}v\right)^{[n]}$
sits inside $\Gamma_{0}=\left(\del_{\Gamma}v\right)^{[n+1]}$. Then
$\Gamma_{0}$ is Cohen-Macaulay of dimension $n+1$, hence $\tilde{H}_{i}(\Gamma_{0})=0$
for all $i\leq n$. The result follows by Lemma \ref{lem:Lift}.
\end{proof}
We take a brief aside to provide a more geometric proof of Corollary
\ref{cor:LiftSCM} when $\del_{\Gamma}v$ satisfies the stronger condition
of shellability. We recall that a shellable complex is built up by
inductively attaching facets, in such a way that the homotopy type
changes only when a facet is attached along its entire boundary. The
following proposition will allow us to give a homotopy type version
of Corollary \ref{cor:LiftSCM} in this broad special case.
\begin{prop}
\label{prop:SheddingLinkAttachment} If $\del_{\Gamma}v$ is a shellable
complex and $v$ is a shedding vertex of $\Gamma$, then $\link_{\Delta}v$
sits inside a contractible subcomplex $\Gamma_{0}$ of $\del_{\Gamma}v$.\end{prop}
\begin{proof}
Let $\mathcal{F}$ be the set of facets in the shelling of $\del_{\Gamma}v$
which attach along their entire boundary, and let $\Gamma_{0}=(\del_{\Gamma}v)\setminus\mathcal{F}$.
Since $\link_{\Delta}v$ contains no facets of $\del_{\Gamma}v$,
we have that $\link_{\Delta}v\subset\Gamma_{0}$. It is a standard
fact in the theory of shellable complexes that $\Gamma_{0}$ is contractible
-- see e.g. \cite[proof of Theorem 4.1]{Bjorner/Wachs:1996}, where
$\Gamma_{0}$ is written with the notation $\Delta^{*}$. 
\end{proof}
With just a little more work, we can recursively compute the homotopy
type of $\Gamma$. Recall that the wedge product $X\wedge Y$ of two
topological spaces $X$ and $Y$ is obtained by identifying some point
in $X$ with some point in $Y$, and let $\susp(\Delta)$ denote the
suspension of $\Delta$. 
\begin{cor}
\label{cor:ShellableLift} If $\del_{\Gamma}v$ is a shellable complex
and $v$ is a shedding vertex of $\Gamma$, then 
\[
\Gamma\simeq\left(\del_{\Gamma}v\right)\wedge\susp\left(\link_{\Gamma}v\right).
\]
 In particular $\tilde{H}_{n+1}(\Gamma)\cong\tilde{H}_{n+1}(\del_{\Gamma}v)\oplus\tilde{H}_{n}(\link_{\Gamma}v)$.\end{cor}
\begin{proof}
This follows immediately by \cite[Lemma 10.4(ii)]{Bjorner:1995},
taking $\Delta_{0}=(\del_{\Gamma}v)\setminus\mathcal{F}$ (as in the
proof of Proposition \ref{prop:SheddingLinkAttachment}), $\Delta_{1}$
to be the complex generated by $\mathcal{F}$, and $\Delta_{2}$ to
be $v*(\link_{\Gamma}v)$.\end{proof}
\begin{rem}
\label{rem:elementaryVDhomtype} If in the statement of Corollary
\ref{cor:ShellableLift} we also have $\link_{\Gamma}v$ to be shellable
(as occurs when $\Gamma$ is vertex-decomposable), then there is a
proof avoiding the somewhat difficult gluing result \cite[Lemma 10.4]{Bjorner:1995},
and using instead more elementary results about homotopy type of shellable
complexes. For in this case, by \cite[Lemma 6]{Wachs:1999b} $\Gamma$
can be shelled by taking the shelling of $\del_{\Gamma}v$ followed
by that of $v*\link_{\Gamma}v$. The special case then follows from
\cite[Theorems 3.4 and 4.1]{Bjorner/Wachs:1996} and a straightforward
computation. 
\end{rem}
We are now ready to prove our theorem:
\begin{proof}[Proof of Theorem \ref{thm:vd.reg.generalized}]
 Suppose that $d=\reg(\link_{\Delta}v)$ . Then by Lemma \ref{lem:RegFromTopology}
there exists a face $\sigma$ of $\link_{\Delta}v$ such that 
\[
\tilde{H}_{d-1}(\link_{\link_{\Delta}v}\sigma)=\tilde{H}_{d-1}(\link_{\Delta}(\sigma\cup v))\neq0.
\]
Every link in a sequentially Cohen-Macaulay complex is also sequentially
Cohen-Macaulay, so in particular $(\link_{\Delta}\sigma)\setminus v=\link_{\Delta\setminus v}\sigma$
is sequentially Cohen-Macaulay. By Corollary \ref{cor:LiftSCM} we
have that $\tilde{H}_{d}(\link_{\Delta}\sigma)\neq0$, hence that
$\reg\Delta\geq\reg(\link_{\Delta}v)+1$, which suffices to prove
the result.
\end{proof}
We remark that the ``dual'' characterization of regularity, as in
part (3) of Lemma \ref{lem:RegFromTopology}, seems to be an essential
part of the proof of Theorem \ref{thm:vd.reg.generalized}, as it
is much easier to understand the structure of links (versus induced
subcomplexes) in a vertex-decomposable complex.
\begin{cor}
\label{cor:ComputComplex} If $\Delta$ is a vertex-decomposable simplicial
complex, then $\reg\Delta$ can be recursively computed. If $\Delta$
has $n$ vertices with $\dim\Delta=d$, then computing $\reg\Delta$
requires computing the homology of at most $O(n^{d})$ subcomplexes.\end{cor}
\begin{proof}
The recursive algorithm is clear. The time bound is because computational
paths involve making at most $n$ choices between the link and deletion,
of which at most $d$ can be ``link''. Hence the number of homology
computations required is $O({n \choose d})=O(n^{d})$.
\end{proof}
The more natural problem from the algebra point of view would be:
given a graph or hypergraph $\H$ with edge ideal $I\subseteq R$,
compute $\reg(R/I)$. Unfortunately, this problem is NP-hard even
for vertex-decomposable graphs, as can be seen by considering a whiskered
graph \cite[Section 4.5]{Woodroofe:2010UNPc}. Since computing the
independence complex of a graph is already an NP-hard problem, the
question of whether $\reg\Delta$ may be efficiently computed from
an appropriate representation of $\Delta$ (e.g., a list of facets)
appears to still be open in general. Corollary \ref{cor:ComputComplex}
settles this question in the affirmative for vertex-decomposable complexes
of fixed dimension such that the shedding order can be efficiently
computed.

\section{\label{sec:Packing-bound}A packing bound on regularity}

In this section, we prove Theorem \ref{intro.thml.packing}. We begin
by observing:
\begin{lem}
\label{lem:CoregVertex} For any simplicial complex $\Delta$, we
have 
\[
\max\{\reg(\link_{\Delta}v)\,\vert\, v\in V(\Delta)\}\leq\reg\Delta\leq\max\{\reg(\link_{\Delta}v)\,\vert\, v\in V(\Delta)\}+1.
\]
\end{lem}
\begin{proof}
The lower bound is Lemma \ref{lem:RegShrinksInLinks}. To prove the
upper bound, suppose that $\reg\Delta=d$. Then, by Lemma \ref{lem:RegFromTopology}
part (2), there is a subset $S$ such that $\tilde{H}_{d-1}(\Delta[S])\neq0$.
Without loss of generality, we can take $S$ to be minimal under inclusion,
so that $d=\reg\Delta[S]$, but for any proper subset $T\subset S$
we have $\reg\Delta[T]<d$. Thus, for any $v\notin S$ we have $\reg(\del_{\Delta[S]}v)=\reg(\Delta[S\setminus v])<d$.
By Lemma \ref{lem:RegAlternative} it then follows that $d=\reg\Delta[S]=\reg(\link_{\Delta[S]}v)+1$.
Since $\link_{\Delta[S]}v=(\link_{\Delta}v)[S]$, we get that $\reg(\link_{\Delta}v)\geq d-1$,
as desired.
\end{proof}
In plain language, Lemma \ref{lem:CoregVertex} says that we can find
a vertex $v$ of $\Delta$ such that the regularity of the link of
$v$ is at most one less than that of $\Delta$.
\begin{rem}
\label{Rem:MoradiKianiShell} Lemma \ref{lem:CoregVertex} was shown
for shellable complexes in \cite[Theorem 2.4]{Moradi/Kiani:2010}.
\end{rem}
For a vertex $v$ in simplicial complex $\Delta$, we let the \emph{degree}
of $v$, denoted $\deg v$, be the degree of $v$ in the corresponding
hypergraph (i.e., the number of edges in $\H(\Delta)$ containing
$v$). Equivalently, $\deg v$ is the number of minimal non-faces
of $\Delta$ containing $v$. Thus, a vertex has degree zero if and
only if it is contained in every maximal face, i.e., if and only if
$\Delta$ is a cone over $\del_{\Delta}v=\link_{\Delta}v$.

Since degree 0 vertices can be deleted without affecting regularity,
the next lemma follows immediately:
\begin{lem}
\label{lem:DegNonzeroRemoval} For any simplicial complex $\Delta$,
we can find a vertex $v$ of non-zero degree such that $\reg\Delta\leq\reg(\link_{\Delta}v)+1$.
\end{lem}
By repeated applications of Lemma \ref{lem:DegNonzeroRemoval}, we
obtain:
\begin{thm}
\label{thm:WeakPackingBound} For any simplicial complex $\Delta$,
we have $\reg\Delta$ to be at most the maximum size of a minimal
face $\sigma$ with the property that $\link_{\Delta}\sigma$ is a
simplex.
\end{thm}
In the case where $\Delta$ is the independence complex of a graph,
we can do somewhat better:
\begin{lem}
\label{lem:DegGr1Removal} Suppose that $G$ is a graph having no
isolated edges, and let $\Delta=\Delta(G)$. Then we can find a vertex
$v$ with $\deg v>1$ such that $\reg\Delta\leq\reg(\link_{\Delta}v)+1$.\end{lem}
\begin{proof}
If $G$ has no vertices of degree 1, then the result follows by Lemma
\ref{lem:CoregVertex}. Otherwise, let $x$ be a vertex of degree
1, and let $y$ be the unique neighbor of $x$ in $G$. We have $\link_{\Delta}x=\Delta(G\setminus\{x,y\})$,
hence 
\[
\reg(\link_{\Delta}x)=\reg(\Delta(G\setminus x\setminus y))=\reg(\Delta(G\setminus y))=\reg(\del_{\Delta}y),
\]
where the second equality follows because $x$ is an isolated vertex
in $G\setminus y$. Then by Lemma \ref{lem:RegAlternative} there
are two possibilities:
\begin{casenv}
\item $\reg(\link_{\Delta}x)+1=\reg(\del_{\Delta}y)+1=\reg\Delta$.

Then by Lemma \ref{lem:RegAlternative} we have $\reg\Delta=\reg(\link_{\Delta}y)+1$,
and we take $v=y$.\medskip{}

\item $\reg(\del_{\Delta}x)=\reg\Delta$.

If $G\setminus x$ has no isolated edge, then there is a vertex $v$
with the desired properties by induction on the number of vertices.

If $G\setminus x$ has an isolated edge, then necessarily this edge
has the form $\{y,z\}$ for some vertex $z$. Then $\reg(\Delta)=\reg(\del_{\Delta}x)=\reg(\link_{\del_{\Delta}x}y)+1\leq\reg(\link_{\Delta}y)+1$,
and $y$ is the desired vertex. (Alternately, the same follows by
observing that $x$, $y$, and $z$ form a connected component isomorphic
to a path of length 3, together with the fact that regularity adds
over connected components.) \qedhere

\end{casenv}
\end{proof}

\begin{proof}[Proof of Theorem \ref{intro.thml.packing}]
 We build up a set $\sigma$ of the centers of center-separated stars
recursively. Begin with $\sigma=\emptyset$. At each step, Lemma \ref{lem:DegGr1Removal}
provides us a vertex $v$ of degree at least 2. We delete $v$ and
its neighbors (since $\link_{\Delta(G)}v=\Delta(G\setminus N[v])$),
add $v$ to $\sigma$, and set aside any isolated edges so created.
We repeat this process until no vertices with degree $\geq2$ remain.

By construction, the star at the vertex $v$ chosen at some step is
center-separated from the stars at vertices chosen in any earlier
steps. Thus, the stars centered at the vertices of $\sigma$ form
a center-separated star packing $\mathcal{P}$. The packing is maximal,
since the recursion terminated when no nondegenerate stars remained.

Moreover, $\link_{\Delta(G)}\sigma=\Delta(G\setminus N[\sigma])$
is the independence complex of a subgraph of $G$ consisting of the
$\ell$ isolated edges set aside during the process, together with
some number of isolated vertices. In particular, $\reg(\link_{\Delta(G)}\sigma)=\ell$.
Then Lemma~\ref{lem:DegGr1Removal} gives the desired inequality
\[
\reg(R/I)=\reg\Delta(G)\leq\reg(\link_{\Delta}\sigma)+\vert\sigma\vert=\ell+\vert\sigma\vert=\zeta_{\mathcal{P}}\leq\zeta(G).\qedhere
\]
\end{proof}
\begin{rem}
\label{Rem:MoradiKianiVD} The parameter $\zeta(G)$ is clearly at
most the parameter $a'(G)$ used in \cite[Theorem 2.1]{Moradi/Kiani:2010};
and Theorem \ref{intro.thml.packing} does not require vertex-decomposability,
as their result does.\end{rem}
\begin{example}
It is instructive to examine the graph $G$ pictured in Figure \ref{fig:PackingExample}.
Because $G$ is chordal, $\reg(R/I)=\indmatch(G)=2$ \cite[Corollary 1.7]{Ha/VanTuyl:2008}.
Our star-packing invariant $\zeta(G)$ is 2 for this graph, achieved
by taking the star at $u$ (leaving an isolated edge). Thus, $\zeta(G)=\reg(R/I)$
here. We remark that this example shows that the minimax version of
$\zeta(G)$ (i.e., minimum size of a maximal star packing, plus the
number of leftover edges) will not bound regularity, as taking a star
at $w$ will show.

\begin{figure}
\includegraphics{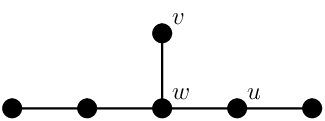}\caption{\label{fig:PackingExample} A graph $G$ with $\zeta(G)=\reg(I/R)=2$,
smaller than $\alpha(G)=3$.}
\end{figure}

The bound from Theorem \ref{thm:WeakPackingBound} is on the other
hand 3, as can be achieved by taking all the vertices of degree 1.
Since the independence number $\alpha(G)$ is also 3, and since $\reg(R/I)$
is obviously at most $\alpha(G)$, the latter bound is in this case
trivial. There are situations where the bound from Theorem \ref{thm:WeakPackingBound}
is nontrivial: for example, if we expand $G$ by adding a pendant
at $v$, then the resulting graph $K$ still has bound from Theorem
\ref{thm:WeakPackingBound} equal to $3$, although $\alpha(K)=4$.
\end{example}

\section*{Acknowledgements}

We thank Ed Swartz for pointing out the elementary proof of the shellable
case of Corollary \ref{cor:ShellableLift} sketched in Remark \ref{rem:elementaryVDhomtype}.
We also thank the anonymous referee for his or her useful comments.

\bibliographystyle{hamsplain}
\bibliography{2_Users_russw_Documents_Research_mypapers_Resul___arity_of_square-free_monomial_ideals_Master}

\end{document}


\global\long\def\E{\mathcal{E}}

\global\long\def\H{\mathcal{H}}

\global\long\def\Y{\mathcal{Y}}

\global\long\def\reg{\operatorname{reg}}

\global\long\def\tor{\operatorname{Tor}}

\global\long\def\indmatch{\nu_{\mathrm{ind}}}

\global\long\def\mmmatch{\nu_{\mathrm{min}}}

\global\long\def\link{\operatorname{link}}

\global\long\def\del{\operatorname{del}}

\global\long\def\susp{\operatorname{susp}}

\title[Corrigendum]{Corrigendum to ``Results on the regularity of square-free monomial
ideals'' {[}Adv. in Appl. Math. 56 (2014), 21\textendash 36{]}}

\author{Huy Tài Hà}

\address{Tulane University \\
 Department of Mathematics \\
 6823 St. Charles Ave. \\
 New Orleans, LA 70118, USA}

\email{tai@math.tulane.edu}

\urladdr{\url{http://www.math.tulane.edu/~tai/} }

\author{Russ Woodroofe}

\address{Department of Mathematics \& Statistics\\
Missisissippi State University, MS 39762, USA}

\email{rwoodroofe@math.msstate.edu}

\urladdr{\url{http://rwoodroofe.math.msstate.edu/}}

\maketitle
The purpose of this short note is to correct two errors in our paper
\cite{Ha/Woodroofe:2014}. The first error is minor. Lemma~3.3 of
the paper should read as follows. (We left off the ``$\left|E\right|$''
in the paper.)
\begin{lem}[{Corrected statement for \cite[Lemma 3.3]{Ha/Woodroofe:2014}}]
\label{lem:Lemma33} Let $\H$ be a hypergraph with at least two
edges, $E$ be an edge of $\H$, and $\H_{E}$ be as in \cite{Ha/Woodroofe:2014}.
Then 
\[
\reg(I(\H))\leq\max\left\{ \reg(I(\H\setminus E)),\reg(I(\H_{E})-1),\left|E\right|\right\} .
\]
\end{lem}
The second error is a substantive mistake in the proof of \cite[Lemma 3.4]{Ha/Woodroofe:2014}.
Fahimeh Khosh-Ahang and Somayeh Moradi have pointed out in \cite[Example 3.8]{Khosh-Ahang/Moradi:2016UNP}
that (in the notation of \cite[proof of Lemma 3.4]{Ha/Woodroofe:2014})
the set $E\cup E_{0}$ may not be an edge. We are grateful to them
for pointing out our error. We shall now give a corrected proof for
Lemma 3.4. All the main results of our paper \cite{Ha/Woodroofe:2014}
are true as stated.

We will use the same notation and conventions as \cite{Ha/Woodroofe:2014};
see also \cite{Ha:2014} for additional background. For the convenience
of the reader, we give a self-contained restatement of the lemma in
question. 
\begin{lem}[Lemma 3.4 of \cite{Ha/Woodroofe:2014}]
 \label{lem:CorrectedLemma} If $\H=(V,\E)$ is a simple hypergraph
with an edge $E_{0}$ such that every edge $E$ has a vertex $v$
with $E\setminus\{v\}\subset E_{0}$, then $\reg(I(\H))=\left|E_{0}\right|$.
\end{lem}
It is useful to notice the trivial upper bound 
\begin{equation}
\reg(I(\H))\leq\left|V\right|\text{ for any hypergraph }\H=(V,\E).\label{eq:TrivBd}
\end{equation}
The following fact improves (\ref{eq:TrivBd}) slightly: 
\begin{lem}
\label{lem:Bd} If $\H=(V,\E)$ is a hypergraph, then $\reg(I(\H))<\left|V\right|$
unless $\E=\left\{ V\right\} $.
\end{lem}
Although we believe Lemma~\ref{lem:Bd} to be well-known, we did
not find it in the literature. We give two proofs, the first of which
avoids induction, the second of which avoids using too much machinery.
\begin{proof}[Proof 1.]
 This is equivalent via \cite[Lemma 2.3]{Ha/Woodroofe:2014} to the
fact from algebraic topology that the only homology $(n-2)$-cycle
on $n$ vertices is the simplex boundary. Suppose that $\Delta=\Delta(\H)$
is a simplicial complex on $n$ vertices with $\tilde{H}_{n-2}(\Delta)\neq0$,
and let $\Gamma$ be the boundary of the $(n-1)$-simplex, which of
course is an $(n-2)$-sphere. But now by Alexander duality \cite[Theorem 3.44]{Hatcher:2002}
we have that $\tilde{H}_{n-2}(\Delta)\cong\tilde{H}^{-1}(\Gamma\setminus\Delta)$,
hence that $\Delta=\Gamma$.
\end{proof}

\begin{proof}[Proof 2.]
 Assume that $\H$ has at least 2 edges, one of which is $E$. By
\cite[Lemma 3.3]{Ha/Woodroofe:2014} and/or Lemma~\ref{lem:Lemma33},
we see that $\reg(I(\H))\leq\max\left\{ \reg(I(\H\setminus E),\reg(I(\H_{E}))-1,\left|E\right|\right\} $,
where $\H_{E}$ is a certain hypergraph on vertex set $V$. The result
now follows by induction and (\ref{eq:TrivBd}).
\end{proof}
We are now ready to correct the proof of Lemma~\ref{lem:CorrectedLemma}.
We proceed by double induction on $\left|\E\right|$ and $\left|V\right|-\left|E_{0}\right|$.
If $\left|\E\right|=1$, then $V=E_{0}$, and the lemma follows from
(\ref{eq:TrivBd}). If $\left|V\right|-\left|E_{0}\right|=0$, then
$\left|\E\right|=1$, and we are in the previous situation.

Thus, for the inductive step, we can take $E_{1}$ to be some edge
of $\H$ other than $E_{0}$. Let $z$ be the unique vertex in $E_{1}\setminus E_{0}$.
Now take $\H_{c}$ (and $\H_{d}$) to respectively be the subhypergraphs
of $\H$ consisting of the edges that contain $z$ (do not contain
$z$). Let $I_{c}=I(\H_{c})$ and $I_{d}=I(\H_{d})$ be the corresponding
ideals.

Since we have partitioned the edges, we see that $I(\H)=I_{c}+I_{d}$,
and so there is a natural short exact sequence
\[
0\to I_{c}\cap I_{d}\to I_{c}\oplus I_{d}\to I_{c}+I_{d}=I(\H)\to0.
\]
 It follows from a standard long exact sequence argument that, in
order to prove the assertion, it suffices to show that $\reg(I_{c}),\reg(I_{d})\leq\left|E_{0}\right|$
and that $\reg(I_{c}\cap I_{d})\leq\left|E_{0}\right|+1$. 

We first observe that $\reg(I_{d})\leq\left|E_{0}\right|$ follows
by induction, as $E_{0}$ satisfies the inductive hypothesis in $\H_{d}$
and $\H_{d}$ has fewer edges than $\H$.

We next observe that every edge of $\H_{c}$ is contained in $E_{0}\cup\{z\}$
by hypothesis. Since isolated vertices do not affect regularity, we
may view $\H_{c}$ as a hypergraph on $\left|E_{0}\right|+1$ vertices.
Since $E_{0}$ is an edge of $\H$, the set $E_{0}\cup\{z\}$ is not
an edge of $\H_{c}$. Now $\reg(I_{c})\leq\left|E_{0}\right|$ by
Lemma~\ref{lem:Bd}.

Finally, we observe that the non-redundant generators of $I_{c}\cap I_{d}$
correspond to the minimal subsets of the form 
\[
\left\{ E\cup F\,:\,E\in\H_{c},F\in\H_{d}\right\} .
\]
 The set $F_{0}=E_{0}\cup E_{1}=E_{0}\cup\{z\}$ has this form, and
is minimal since any $F\in\H_{d}\setminus E_{0}$ has at least one
vertex in $V\setminus E_{0}\setminus\{z\}$. Moreover, it is clear
that $F_{0}$ satisfies the inductive hypothesis in the hypergraph
associated with $I_{c}\cap I_{d}$. Since $\left|V\right|-\left|F_{0}\right|<\left|V\right|-\left|E_{0}\right|$,
induction now gives $\reg(I_{c}\cap I_{d})\leq\left|F_{0}\right|=\left|E_{0}\right|+1$.
This completes the proof. 

\bibliographystyle{hamsplain}
\bibliography{6_Users_russw_Documents_Research_mypapers_Results_on_reg_corrigendum_Master}